\documentclass[reqno,12pt]{amsart}
\usepackage{mathtools,enumerate}
\usepackage{amssymb,latexsym}
\usepackage{fullpage}
\usepackage{color}
\theoremstyle{plain}

\theoremstyle{remark}

\numberwithin{equation}{section}
\usepackage{amscd} 

\usepackage[makeroom]{cancel}
\usepackage[all,knot]{xy} 
\usepackage{multicol} 
\usepackage{verbatim} 
\usepackage{mathtools}  
\usepackage{bookmark} 
\usepackage{ytableau}
\usepackage{microtype}
\usepackage{bbold}  
\usepackage{tikz}
\usetikzlibrary{decorations.markings}
 \usepackage{graphicx}
\usepackage[all]{xy}
\usepackage{xypic}

\swapnumbers
\usepackage{lineno}
\usepackage{tikz-cd}
\newtheorem{thm}{Theorem}[subsection]

\newtheorem{lem}[thm]{Lemma}
\newtheorem{cor}[thm]{Corollary}
\newtheorem{conj}[thm]{Conjecture}
\newtheorem{prop}[thm]{Proposition}

\newcommand{\germ}{\mathfrak}

\newcommand{\End}{\mathop{\mathrm{End}}\nolimits}
 
\newcommand{\tiep}{T_{i,e}'}
\newcommand{\tiepp}{T_{i,e}''}
\newcommand{\qbinom}[2]{\genfrac{[}{]}{0pt}{}{#1}{#2}}
\newcommand{\angles}[2]{\langle{#1},{#2}\rangle}
\newcommand{\bui}{\mathbf U_i}
\newcommand{\bu}{\mathbf U}
\newcommand{\G}{\Gamma}
\newcommand{\ru}{\mathbf U_R}
\newcommand{\ra}{\to}
\newcommand{\brho}{\mathbf \rho}

\newcommand{\corref}[1]{Corollary~\ref{#1}}
\newcommand{\secref}[1]{\S\ref{#1}}
\newcommand{\lemref}[1]{Lemma~\ref{#1}}
\newcommand{\eqnref}[1]{~(\ref{#1})}

\ifx\xyloaded\undefined \input xy \fi
\xyprovide{poly}{Polygon feature}{\stripRCS$Revision: 3.11 $}
\xyrequire{arrow}\xycatcodes
\author{Ben L. Cox}
\address{Department of Mathematics \\
The University of Charleston \\
66 George Street  \\
Charleston SC 29424 \\ FAX: 843-953-1410}
\author{Thomas J. Enright}
\address{Department of Mathematics \\
UC San Diego \\
9500 Gilman Dr. La Jolla, CA 92093}
\title{Representations of quantum groups defined over commutative rings III.
}
\begin{document}

\maketitle
\begin{abstract}  We survey some of our old results given in \cite{MR1327136} and \cite{MR2586983} and present some new ones in the last three sections. 
\end{abstract}   

\section{Introduction and Summary of Results.} 

\subsection{}

For the convenience of the reader we survey below material that was developed in \cite{MR1327136} and \cite{MR2586983}.

Let $v$ be an indeterminate and $\mathbb k$ a field of characteristic zero. Let
$\bu$ be the quantized enveloping algebra defined over $\mathbb k(v)$ with generators
$K^{\pm 1},E,F$ and relations 
$$ 
[E,F]=\frac{K-K^{-1}}{ v-v^{-1}} , \quad KEK^{-1}=v^2E\quad
\text{and} \quad KFK^{-1}=v^{-2}F.
$$ Let $\bu^0$ be the subalgebra generated by $K^{\pm 1}$ and let $B$ be
the subalgebra generated by $\bu^0$ and $E$. More precisely we are
following the notation given in \cite{MR2586983} where we take
$I=\{i\}$, $i\cdot i=2$, $Y=\mathbb Z[I]\cong \mathbb Z$, 
$X=\hom(\mathbb Z[I], \mathbb Z)\cong \mathbb Z$, $F=F_i$, $E=E_i$, and
$K=K_i$.

Let $R$ be the power series ring in $T-1$ with coefficients in $\mathbb k(v)$ i.e. 
\begin{equation}
R=k(v)[[T-1]]:=\lim_\leftarrow \frac{\mathbb k(v)[T,T^{-1}]}{ (T-1)^i}.
\end{equation}
Set $\mathcal K$ equal to the field of fractions of $R$. Let $s$ be the
involution of $R$ induced by $T\ra T^{-1}$, i.e. the involution that sends
$T$ to $T^{-1}=1/(1+(T-1))=\sum_{i\geq 0}(-1)^i(T-1)^i$. Let the subscript
$R$ denote the extension of scalars from $\mathbb k(v)$ to $R$ , e.g. $\ru =
R\otimes_{\mathbb k(v)} \bu$. For any representation $(\pi,A)$ of $\ru$ we can
twist the representation in two ways by composing with automorphisms of
$\ru$. The first is $\pi \circ (s\otimes 1)$ while the second is $\pi
\circ (1\otimes \Theta)$ for any automorphsim $\Theta$ of $\bu$. We designate the corresponding
$\ru$-modules by $A^s$ and $A^{\Theta}$. Twisting the action by both $s$
and $\Theta$ we obtain the composite $(A^s)^{\Theta}=(A^{\Theta})^s$ which
we denote by
$A^{s\Theta}$.

Let $m$ denote the homomorphism of $\ru^0$ onto $R$ with $m(K)=T$. For
$\lambda \in \mathbb Z$ let $m+\lambda$ denote the homomorphism of
$\bu^0$ to
$R$ with $(m+\lambda )(K)=Tv^\lambda$. We use the additive notation
$m+\lambda$ to indicate that this map originated in the classical setting
from an addition of two algebra homomorphisms. It however is not a sum of
two homomorphisms but rather a product. Let
$R_{m+\lambda}$ be the corresponding $_RB$-module and define the Verma
module
 \begin{equation}
_RM(m+\lambda )=\ru\otimes_{_RB} R_{m+\lambda-1} .
\end{equation}

Let $\rho_1:\mathbf U\to \mathbf U$ be the algebra isomorphism
determined by the assignment 
\begin{equation}
\rho_1(E)=-vF,\quad \rho_1(F)=-v^{-1}E,\quad
\rho_1(K)=K^{-1} 
\end{equation} 
for all $i\in I$ and $\mu\in Y$.   Define also an algebra
anti-automorphism $\brho:\mathbf U\to \mathbf U$ by 
\begin{equation}
\brho(E)=vKF,\quad \brho(F)=vK^{-1}E,\quad
\brho(K_\mu)=K_\mu. 
\end{equation}
These maps are related through the antipode $S$ of  $\mathbf U$ by
$\brho=\brho_1S$. 

For $\ru$-modules $M,N$ and $\mathcal F$, let $\mathbb P(M,N)$ and $\mathbb
P(M,N,\mathcal F)$ denote the space of $R$-bilinear maps of $M\times N$ to $R$
and $\mathcal F$ respectively, with the following invariance condition: 
\begin{equation}
\sum x_{(1)}*\phi(Sx_{(3)}\cdot a,\varrho(x_{(2)})b)=\mathbf e(x) 
\phi(a,b)
\end{equation} where $\Delta\otimes 1\circ \Delta (x)=\sum x_{(1)}\otimes
x_{(2)}\otimes x_{(3)}$ and $\mathbf e:\bu\to k(v)$ is the counit.  If we let
$\hom_{\ru}(A,B)$ denote the set of module $\ru$-module homomorphisms, then one
can check on generators of $\ru$ that
$\mathbb P(M,N,\mathcal F)\cong \hom_{\ru} (M\otimes_R N^{\rho_1},{_R\mathcal
F}^\brho)$ (see \cite[3.10.6]{MR1359532}).  Let $\mathbb P(N)=\mathbb
P(N,N)$ denote the
$R$-module of invariant forms on $N$. 

For the rest of the introduction we let $M$ denote the $\ru$ Verma
module with highest weight $Tv^{-1}$ i.e. $M=M(m)$ and let $\mathcal
F$ be any finite dimensional $\bu$-module.  A natural parameterization for
$\mathbb P(M\otimes {_R\mathcal F})$ was given in \cite{MR1327136}.
Fix an invariant form $\phi_M$ on $M$ normalized as in \eqnref{}. For each
$\ru$-module homomorphism $\beta:{_R\mathcal F}\otimes_R\mathcal
F^{\rho_1}\ra\ru$ define what we call the induced form
$\chi_{\beta,\phi_M}$ by the formula, for
$e,f\in _R\mathcal F, \ m,n\in M$,
\begin{equation}
\chi_{\beta,\phi_M}(m\otimes e,n\otimes f)=\phi_M( m,\beta(e\otimes f)*n).
\end{equation}
\begin{prop}[\cite{MR2586983}]
Suppose $\beta:{_R\mathcal F}\otimes {_R\mathcal F^{\rho_1}}\ra \ru$ is a module
homomorphism with $\ru$ having the adjoint action.    Then $M\otimes
{_R\mathcal F}$ decomposes as the $\chi_{\beta,\phi_M}$-orthogonal sum of
indecomposible $\ru$-modules.
\end{prop}

This last result has a number of intriguing consequences which are
formulated in the context of induced filtrations.
For any
$R$-module $B$ set $\overline B=B/(T-1)\cdot B$ and for any filtration
$B=B_0\supset B_1\supset ...\supset B_r$,  let 
$\overline B=\overline B_0\supset \overline B_1\supset ...\supset
\overline B_r$ be the induced filtration of $\overline B$, with
$\overline B_i=(\overline B_i+(T-1)\cdot B)/ (T-1)\cdot B$.

Now an invariant form $\chi$ on an $\ru$-module $B$ gives a
filtration on
$B$ by setting
\begin{equation}\label{filtration}
B_i=\{v\in B | \chi(v,B)\subset
(T-1)^i\cdot R\}.
\end{equation}
\begin{prop}[\cite{MR2586983}]
Suppose $\mathcal F$ is a finite dimensional $\bu$-module and
$\phi$ is an invariant form on $M\otimes {_R\mathcal F}$. Let
$M\otimes {_R\mathcal F}=B_0\supset B_1\supset ...\supset B_r$ be the
filtration \eqnref{filtration} and 
$\overline B_0\supset \overline B_1\supset ...\supset
\overline B_r$ the induced filtration on $\overline{M\otimes {_R\mathcal F}}$.
Then
\begin{enumerate}
\item  The $\bu$-module $\overline{B_i}/\overline{B_{i+1}}$
is finite dimensional for $i$ odd.
\item  The $\bu$-module $\overline B_i/\overline B_{i+1}$ is
both free and cofree as a $\mathbb k(v)[F]$-module, for $i$ even.
\end{enumerate}
\end{prop}

A final result is cast in the
language of hereditary filtrations which we now describe.  Let
$\mathcal E$ denote the simple two dimensional
$\bu$-module with highest weight $v$ and set $P=M\otimes {_R\mathcal E}$. Then
$P$ is isomorphic to the basic module $P_1=P(m+1)$ as defined through the equations \eqnref{P_iF}. Set
$M_\pm=M(m\pm 1)$. The construction of $P$ gives
inclusions and a short exact sequence:
\begin{equation}\label{ses-inclusions}
(T-1)M_+\oplus (T-1)M_- \subset P\subset
M_+\oplus M_- ,\quad 0\ra (T-1) M_+\ra P\ra M_- \ra 0.
\end{equation}

Let $\mathcal F$ be any finite dimensional $\bu$-module and set 
\begin{equation}\label{mathbbABDP}
\mathbb A=(T-1)\cdot M_+\otimes _R\mathcal F ,\quad \mathbb B=(T-1)\cdot
M_- \otimes _R\mathcal F ,\quad \mathbb D=
M_- \otimes_R\mathcal F ,\quad \mathbb P=P\otimes_R\mathcal F.
\end{equation}
 Then \eqnref{ses-inclusions},
gives inclusions and the short exact sequence:
\begin{equation}
\mathbb A\oplus \mathbb B\subset \mathbb P \subset (T-1)^{-1}(\mathbb A\oplus \mathbb B),\quad\quad 0\to
\mathbb A\to \mathbb P\ { \overset \Pi\to}\ \mathbb D\to 0 .
\end{equation}

Now fix an induced invariant form $\chi=\chi_{\beta,\phi_P}$ on $\mathbb P$ where $\beta$
is a homomorphism of $_R\mathcal F\otimes_R \mathcal F^{\rho_1}$ into
$\ru$ and 
$\phi_P$ is an invariant form on $P$. Filter $\mathbb A,\mathbb B$ and $\mathbb P$ using \eqnref{filtration} and 
\eqnref{mathbbABDP}. Then we say that $\chi$ gives a {\it hereditary filtration} if,
for all i,
\begin{equation}
\mathbb A_i\cap
\mathbb B=\mathbb B_i.
\end{equation}
For any weight module $N$ for $\bu$, let $N_{\le h}$ denote the 
span of the weight spaces with weights $t\le h$. We say that 
$\chi$ gives a {\it weakly hereditary filtration} if, for some $h$
and for all $i$,
$\mathbb A_{i,\le h}\cap \mathbb B=\mathbb B_{i,\le h}$.

There is an action of the Weyl $\{1,s\}$ group of $\germ{sl}_2$ on the space of forms $\chi_{\beta,\phi}$ where the lifted form $\chi_{\beta,\phi}^\sharp$ satisfies $\chi_{\beta,\phi}^\sharp=\chi_{s\beta,\phi}$ (see \cite[Theorem 38]{MR2586983}).  We say that $\chi$ is even if  $\chi_{\beta,\phi}^\sharp=\chi_{\beta,\phi}$ and odd if  $\chi_{\beta,\phi}^\sharp=-\chi_{\beta,\phi}$.  Our aim is to prove our conjecture that
\begin{conj}
Suppose $\chi$ is either even or odd.
Then $\chi$ gives a weakly hereditary filtration.
\end{conj}

\medskip
\section{$q$-Calculus.}
\subsection{Definitions}
As many before us have done, we define
\begin{align*}
[m]&:=\frac{v^m-v^{-m}}{v-v^{-1}},  \\
[m]!&:=[m]\cdot [m-1]\cdots [1] \\
[0]!&:=1  \\
\qbinom{m}{ n}&=\frac{[m]!}{[n]![m-n]!}\quad \text{for} \quad n\leq m 
\\ \\
\qbinom{m}{ n}&=\begin{cases}0 &\quad {\rm if}\quad m<n\quad{\rm
or}\quad n<0 , \\
1 &\quad {\rm if}\quad n=m\quad{\rm
or}\quad m=0. 
\end{cases}
\end{align*}

For $r\in Z$ define
\begin{gather}
[T;r]:=\frac{v^{r}T- v^{-r}T^{-1}}{ v-v^{-1}},   \\ \\
[T;r]_{(j)}:=[T;r][T;r-1]\cdots [T;r-j+1],
\quad\quad [T;r]^{(j)}:=[T;r+1]\cdots [T;r+j]
\quad\text{if}\quad j>0,\notag\\ 
[T;r]_{(0)}:= [T;r]^{(0)}:=1,
\notag \\ \notag \\
\qbinom{T;r}{ j}:=\begin{cases}[T;r]_{(j)}/[j]!
& \text{if}\quad j\geq 0 \\
0&  \text{if}\quad j<0.\end{cases}
\end{gather}
Observe that 
\begin{equation}\label{r1}
[T;r]^{(j)} = [T;r+j]_{(j)}.
\end{equation}
and 
\begin{equation}\label{r2}
(-1)^j[T;r]^{(j)}=[T^{-1};-r-1]\cdots
[T^{-1};-r-1-j+1] =  [T^{-1};-r-1]_{(j)}.
\end{equation}
Moreover note that $[T;\lambda]^{(k)}$ is invertable in $R$ for $\lambda\geq
0$ and $[T;\lambda+1]_{(k)}$ is invertible provided
and $\lambda+1> k$ or $\lambda<0$ ($k\geq 0$).  In fact
\begin{equation}\label{cong1}
[r]![T;r]_{(r)}^{-1}\cong T^r\mod (T-1).
\end{equation}
Indeed the map sending $T\mapsto 1$ defines a surjection of $R$ onto $\mathbb k(v)$
and under this map $[T;r]_{(r)}\mapsto [r]!$. Moreover for $k\neq 0$,
\begin{align*}
[T;k]^{-1}
&=T[k]^{-1}\left(
\frac{1}{1-(T-1)\frac{v^{k}}{1-v^k}}\right)\left(
\frac{1}{1+(T-1)\frac{v^{k}}{1+v^k}}\right) 
\end{align*}
and the two fractions on the right can be written as power series in $T-1$
with $1$ as their leading coefficient.
\color{black}

\subsection{Identities}
Two useful formulae for us will be

\begin{equation}\label{Ma1}
\qbinom{s-u}{ r}=\sum_p
(-1)^pv^{\pm (p(s-u-r+1)+ru)}\qbinom{u}{ p}\qbinom{s-p}{
r-p}
\end{equation}
\begin{equation}\label{Ma2}
\qbinom{u+v+r-1}{ r}=\sum_p
v^{\pm (p(u+v)-ru)}\qbinom{u+p-1}{ p}\qbinom{v+r-p-1}{r-p}
\end{equation}
which come from \cite{MR1337274}, equations 1.160a and 1.161a respectively.

\section{$\bui$ Automorphisms and Intertwining Maps.} 
\label{automorphism}
\subsection{} Following Lusztig, \cite[Chapter
5]{MR1227098}, we let
$\mathcal C'$ denote the category whose objects are $\mathbb Z$- graded
$\bu$-modules $M=\oplus _{n\in\mathbb Z}M^n$ such that 

\begin{enumerate}[(i)]
\item  $E,F$ act locally nilpotently on $M$,
\item  $Km=v^nm$ for all $m\in M^n$.
\end{enumerate}

Fix $e=\pm 1$ and let $M\in \mathcal C'$.  Define Lusztig's automorphisms
$  T_i',  T_i'':M\to M$ by
\begin{equation}
 T_i'(m):=\sum_{a,b,c;a-b+c=n}(-1)^bv^{e(-ac+b)}F^{(a)}E^{(b)}F^{(c)}m, 
\end{equation}
and
\begin{equation}
 T_i''(m):=\sum_{a,b,c;-a+b-c=n}(-1)^bv^{e(-ac+b)}E^{(a)}F^{(b)}E^{(c)}m
\end{equation}
for $m\in M^n$.  In the above $E^{(a)}:=E^a/[a]!$ is the $a$-th {\it
divided power of $E$}.

Lusztig defined automorphisms $T''_{i,e}$ and $ T'_{i,e}$ on $\bu$ by
$$
T'_e(E^{(p)})=(-1)^pv^{ep(p-1)}K^{ep}F^{(p)},\quad \quad
T'_e(F^{(p)})=(-1)^pv^{-ep(p-1)}E^{(p)}K^{-ep}
$$
and
$$
T''_{-e}(E^{(p)})=(-1)^pv^{ep(p-1)}F^{(p)}K^{-ep},\quad \quad
T''_{-e}(F^{(p)})=(-1)^pv^{-ep(p-1)}K^{ep}E^{(p)}.
$$

One can check on generators that $\rho_1\circ T_{-1}'=T_{-1}'\circ
\rho_1$.  
In order to distinguish the algebra homomorphisms above
from their module homomorphism counterparts we will sometimes use the
notation 
$T_{e,\rm{mod}}'$ and  $T_{e,\rm{mod}}''$to denote the later.  If
$M$ is in
$\mathcal C'$,
$x\in \bu$ and $m\in M$, then we have

\begin{equation}
\Theta (x\cdot m) =\Theta(x) \Theta m 
\end{equation}
for $\Theta=\tiep$ or $\Theta=\tiepp$ (see \cite[37.1.2]{MR1227098}). The last identity can be interpreted to say that
$\Theta$ and
$\Theta\otimes s$ are intertwining maps;

\begin{equation}
\Theta:M \ra M^\Theta\quad \quad \quad \quad
\Theta\otimes s\ :{_R}M\ra {_R}M^{\Theta\otimes s}.
\end{equation}

To simplify notation we shall sometimes write $s\Theta$
in place of $\Theta\otimes s$. \smallskip

We now describe the explicit action of $\Theta$ on $M$.
\begin{lem}(\cite[Prop.5.2.2]{MR954661}). \label{symmetries} Let
$m\geq 0$ and
$j,h\in [0,m]$ be such that $j+h=m$.
\begin{enumerate}[(a)]
\item If $\eta\in M^m$ is such that $E\eta=0$, then
$\tiep(F^{(j)}\eta)=(-1)^jv^{e(jh+j)}F^{(h)}\eta$.
\item If $\zeta\in M^{-m}$ is such that $F\zeta=0$, then
$\tiepp(E^{(j)}\zeta) =(-1)^jv^{e(jh+j)}E^{(h)}\zeta$.
\end{enumerate}
\end{lem}

Let $F(\bu)$ denote the ad-locally finite submodule of $\bu$.  We know
from \cite{MR1198203} that $F(\bu)$ is tensor
product of harmonic elements $\mathcal H$ and the center $Z(\bu)$.  Here
$\mathcal H=\oplus_{m\in\mathbb Z} \mathcal H_{2m}$ and $\mathcal
H_{2m}=\text{ad}\, \bu (EK^{-1})$.

 There is another category that
we will need and it is defined as follows: Let
$M$ be a
$\ru$-module.  One says that
$M$ is {\it
$_R\bu^0$-semisimple} if $M$ is the direct sum of $R$-modules
$M^{\mu}$ where $K$ acts by $T \ v^\mu$, $\mu\in\mathbb Z$; i.e. by
weight $m+\mu$.  Then $\mathcal C_{R}$ denotes the category of $\ru$-modules
$M$ for which $E$ acts locally nilpotently and  $M$ is
$_R\bu^0$-semisimple.

For $M$ and $N$ two objects in $\mathcal C'$ or one of them is in 
$_R\mathcal C$,
Lusztig defined
the linear map
$L:M\otimes N\to M\otimes N$ given by
\begin{equation}\label{L}
L(x\otimes y)=\sum_n(-1)^nv^{-n(n-1)/2}\{n\}F^{(n)}x\otimes E^{(n)}y
\end{equation} 
where $\{n\}:=\prod_{a=1}^n(v^a-v^{-a})$ and $\{0\}:=1$.  
One can show
$$ 
L^{-1}(x\otimes y)=\sum_nv^{n(n-1)/2}\{n\}F^{(n)}x\otimes E^{(n)}y.
$$ 

\begin{lem}\label{LLemma1}(\cite{MR1227098}).
Let $M$ and $N$ be two objects in $\mathcal C'$.  Then
$ T_{1}'' L(z)=(T_{1}''\otimes T_{1}'')(z)$ for all $z\in
M\otimes N$. 
\end{lem}

\begin{lem}\label{LLemma2}
Let $M$ be a module in $\mathcal C_R$ and $N$ a module in
$\mathcal C'$.  Then for $x\in M^t$ and $y\in N^s$ we have
\begin{align*}
FL(x\otimes y) &=L(x\otimes Fy+v^sFx\otimes y) \\
EL(x\otimes y) &=L(Ex\otimes y+v^{-t}T^{-1}x\otimes Ey) .
\end{align*}
\end{lem}

\begin{cor}\label{Lisomorphism}
Let $M$ be a module in $\mathcal C_R$ and $\mathcal E$ a module in
$\mathcal C'$.  Then $L$ defines an isomorphism of
the $\bu$-module
$M^{T_{-1}'}\otimes \mathcal E^{T_{-1}'}$ onto $(M\otimes\mathcal
E)^{T_{-1}'}$.
\end{cor}

\corref{Lisomorphism} through the use of \lemref{LLemma1}, however one
must take into account that
$T_e'$ may not be defined on $M$. 

Set 
\begin{equation}\label{mathcalLinv}
\mathcal L^{-1}=\sum_p(-1)^pv^{3\frac{p(p-1)}{
2}}\{p\}E^{(p)}K^pF^{(p)},
\end{equation}
and
\begin{align}\label{mathcalL}
\mathcal L
&=\sum_p
    v^{-3\frac{p(p-1)}{2}}\{p\}E^{(p)}K^{-p}F^{(p)}.
\end{align}
and note that $\mathcal L$ and $\mathcal L^{-1}$ are well
defined operators on lowest weight modules.  

\begin{lem} Suppose that $M$ and $N$ be highest weight
modules with $\psi_M:M^{sT'_{-1}}_\pi\to M$, $\psi_N:N^{sT'_{-1}}_\pi\to
N$  homomorphisms and 
$\phi$ a $\rho$-invariant form on $M\times N$.  Then

\begin{equation}\label{Ll}
\phi\circ(\psi_M\otimes \psi_N)\circ L^{-1}=\phi\circ
(\psi_M\otimes \psi_N) \circ (\mathcal  L^{-1}\otimes 1).
\end{equation}

\end{lem}

\begin{lem} $T'_{-1}(u)\mathcal  L^{-1}=\mathcal  L^{-1}T_1'(u)$ as
operators on $M_\pi$ for all $u\in{_R\mathbf U}$.
\end{lem}

%

\section{Invariant Forms and Liftings} 

\subsection{} Elements of $\mathbb P(M,N)$ are called {\it invariant pairings} and for $M=N
$, set
$\mathbb P(M) =\mathbb P(M,M)$ and call the elements {\it invariant forms} on
$M$.

\begin{lem}
Let $M$ and $N$ be finite dimensional $\ru$-modules in
$\mathcal C'$ and   $\phi$ an invariant pairing. Then, for $m\in M,n\in N$, $
\phi(\tiepp m,\tiepp n)=\phi(m,n)$.
\end{lem}

Note that if $\phi(\eta,\eta)=1$, then for $0\leq j\leq \nu$, the
proof above shows that 
\begin{equation}\label{normalizedforms}
\phi(F^{(j)}\eta,F^{(j)}\eta)=v^{j^2-\nu j}{\qbinom \nu j}
\end{equation}

For the proof of some future results we must be explicit about the 
definition of $_f\mathcal R$.    Recall a
$\bu$-module
$M$ is said to be {\it integrable} if for any $m\in M$ and all $i\in I$, there exists a
positive integer
$N$ such that $E^{(n)}_im=0=F^{(n)}_im$ for all $n\geq N$ , and $M=\oplus
_{\lambda\in X}M^\lambda$ where for any
$\mu\in Y, \lambda\in X$ and $m\in M^\lambda$ one has
$K_\mu m=v^{\angles{\mu ,\lambda}m}$. Let $\bu_0^\times$ denote the set of units of
$\bu_0$ and let
$f:X\times X\to \bu_0^\times$ be a function such that  
 \begin{equation}
f(\zeta+\nu,\zeta'+\nu')=f(\zeta,\zeta')v^{ -\sum
\nu_i\angles{i,\zeta'}(i\cdot i/2)-\sum
\nu'_i\angles{i,\zeta}(i\cdot i/2) -\nu\cdot
\nu'}\tilde K_\nu
\end{equation}
for all $\zeta,\zeta'\in X$ and all $\nu,\nu'\in X$ (see
\cite[32.1.3]{MR1227098}). Here 
$\tilde K_\nu=\prod_iK_{(i\cdot i/2)\nu_ii}$.   

\begin{thm}(\cite[32.1.5]{MR1227098}). If
$\mathcal E$ is an integrable
$\ru $ module and $A\in \mathcal C_{R}$, then for each $f$ satisfying (3.1.1), there
exists an isomorphism $_f\mathcal R:A\otimes
\mathcal E\to \mathcal E\otimes A$.
\end{thm}

The map $\tau:A\otimes B\to B\otimes A$ for any two modules $A$ and $B$ denotes the
twist map
$\tau(a\otimes b)=b\otimes a$.   Define $\prod _f:\in\End_R({_R\mathcal E}\otimes{_R\mathcal F}\otimes_R M)$ by 
$\prod_f(e\otimes e'\otimes  m)= f(\lambda,\lambda')e\otimes e'
\otimes  m$ for $m\in M^{\lambda'}$ and
$e\otimes e'\in (\mathcal E\otimes{_R\mathcal F})^{\lambda}$.  Lastly we define
$\chi\in
\End_R({_R\mathcal E\otimes{_R\mathcal F}}\otimes_R M)$ by
$$
\chi(e\otimes e'\otimes m)=
\sum_\nu\sum_{b,b'\in
\mathbf B_\nu}p_{b,b'}b^-(e\otimes e')\otimes {b'}^+m
$$  where $p_{b,b'}=p_{b',b}\in R$, and $\mathbf B_\nu$ is a subset of
$\mathfrak f$.   Then $_f\mathcal R$ is defined to be equal to $\chi\circ
\prod_f\circ \tau$.  The proof that it is an $\bu$-module homomorphism is
almost exactly the same as in \cite{MR1227098}, 32.1.5, or
\cite{MR1359532}, 3.14], which the exception that one must take
into account that $M$ is in the category
$\mathcal C_R$ instead of $_R\mathcal C'$.

Let $\mathcal E$ and $\mathcal F$ be finite dimensional $\bu$-modules and 
$\tau: _R\mathcal E\otimes _R\mathcal F^{\rho_1}  \rightarrow \bu$, a
$\bu$-module homomorphism into $\bu$, where $\bu$ is a module under
the adjoint action.
Suppose $\phi$
is a pairing of $M$ and $N$.  Define $\psi_{\tau,\phi}$ to be  the
invariant pairing of $ M\otimes _R\mathcal E$ and $N\otimes _R\mathcal F$ defined
by
the formula, for $e\in \mathcal E ,f\in  \mathcal F,m\in M,$ and $n\in N$, 
\begin{equation}\label{inducedpairing}
\psi_{\tau,\phi} (m\otimes e,n \otimes f)=
\phi( m,\tau(e\otimes f)*n)\ ,  
\end{equation}
Here  ${\mathcal F}^{\rho_1}$ is a twist of the representation
$\mathcal F$ by $\rho_1$.  We call the pairing $\psi_{\tau,\phi}$ the
{\it pairing induced} by
$\tau$ and $\phi$. In the cases when $M,N$ and $\phi$ are fixed we write
$\psi_\tau$ in place of $\psi_{\tau,\phi}$ and say this pairing is induced
by $\tau$.

Let us check that $\psi_{\tau,\phi}$ indeed is a $\brho$-invariant
pairing:  For $x\in\bu$ 
\begin{align*}
\sum \psi_{\tau,\phi} (S(x_{(2)})(m\otimes e),&\brho(x_{(1)})(n \otimes
f)) \\ 
&=\sum \psi_{\tau,\phi} (S(x_{(4)})m\otimes
S(x_{(3)})e),\brho(x_{(1)})n
\otimes
\brho(x_{(2)})f)) \\
&=\sum \phi(S(x_{(4)})m,
\tau(S(x_{(3)})e\otimes 
\brho(x_{(2)})f)*(\brho(x_{(1)})n)) \\
&=\sum \phi(S(x_{(4)})m,
\sum \brho(x_{(3)})\rho_1(\tau(e\otimes
f))\brho(x_{(1)}S(x_{(2)}))n)
\\ 
&=\sum \phi(S(x_{(2)})m,
\brho(x_{(1)})(\tau(e\otimes f)*n))) \\
&=\mathbf e(x)\phi( m,\tau(e\otimes f)*n)\ ,  
\end{align*}
The first equality is due to the act that 
\begin{equation}\label{diagonalS}
S\otimes S\circ \Delta =\tau
\circ \Delta \circ S
\end{equation}
where $\tau:\bu\otimes \bu\to \bu\otimes \bu$ is
the twist map and the fact that $\brho\otimes
\brho\circ \Delta=\Delta\circ\brho$.  The second equality follows from
the definition of $\psi_{\tau,\phi}$.  The third equality is obtained
from \eqnref{diagonalS} and the fact that $\brho$ is an
anti-automorphism.  The last equality is due to the assumption that
$\phi$ is $\brho$-invariant.


\subsection{} 
A result from \cite{MR1327136} shows that in the setting of Verma
modules the collection of maps $\tau$ is a natural set of parameters for
invariant forms. 

\begin{prop}[\cite{MR1327136}]
Suppose $\bu$ is of finite type and $\mathcal E$ and
$\mathcal F$ are finite dimensional $\bu$-modules. Let  $M$ be an
$_R\bu$-Verma module  and $\phi$ the Shapovalov form on $M$.
Suppose the Shapovalov form on $M$ is nondegenerate. Then
every invariant pairing of
$M\otimes _R\mathcal E$ and $M\otimes _R\mathcal F$ is induced by $\phi$.
\end{prop}

\subsection{}

\begin{thm}(\cite[Lifting Theorem]{MR1327136}).
Let $A$ and
$B$ be modules in $\mathcal C_{R}$ and $\phi\in {\mathbb P}_{\brho }(A,B)$. Then
$\phi$ uniquely determines an invariant form
$\phi_F\in  {\mathbb P}_{\brho}(A_F,B_F)$ which is determined by the following
properties:
\begin{enumerate}
\item $\phi_F$ vanishes on the subspaces $\iota A\times B_F$ and
$A_F\times \iota B$ .
\item For each $\mu \in \mathbb Z$ with $\mu+1= r\in \mathbb N$, and any vectors $a\in A$ and $b\in B$ both of weight $m+\epsilon\mu$ with
$\epsilon\in
\{1,s\}$ and  $E\ a=E\ b=0$ , 
\begin{equation}
 \phi_F(F^{-1}a ,F^{-1} b)=  v^{-r+1}\frac{\iota\epsilon[T;0]}{  \i
\epsilon[T;r-1]\ }\ \
\phi(a,b).
\end{equation}
\end{enumerate}
\end{thm}

\begin{prop}  The form
$\phi_F$ induces an $\brho$-invariant bilinear map on
$A_\pi \times B_\pi$ which we denote by
$\phi_\pi$.
\end{prop}        

\subsection{} At times the subscript notation for lifted forms will be
inconvenient and so we shall also use the symbol $loc$ for the localization
of both forms and modules. We write $loc(\phi)$ and $loc(A)$ in place of
$\phi_F$ and $A_F$. 

For invariant forms we find that induction and localization commute in the
following sense.

\section{Quantum Clebsch-Gordan decomposition}

\subsection{Basis and Symmetries} \label{basisandsymmetries} For
$m\in\mathbb Z$, let
$\mathcal F_m$ denote the finite dimenisonal irreducible module of
highest weight
$v^m$ with highest weight vector $u^{(m)}$.  For $k$ any non-negative
integer set
$u^{(m)}_k=F^{(k)}u^{(m)}$ and
$u^{(m)}_{-1}=0$.

In particular 
$$
\theta^{-1}(u^{(m)}_j)=T_1''(u^{(m)}_j)=(-1)^{m-j}v^{(m-j)(j+1)}u^{(m)}_{m-j}
$$
and
\begin{equation}\label{irred}
K^pu^{(m)}_j=v^{p(m-2j)}u^{(m)}_j,\quad \quad F^{(p)}u^{(m)}_j=
{\qbinom {p+j} j}u^{(m)}_{j+p},\quad\quad
 E^{(p)}u^{(m)}_j=\qbinom{m+p-j}{
p} u^{(m)}_{j-p}.
\end{equation}

\begin{lem} \label{CG}[Clebsch-Gordan, \cite{MR2586983}] For any two non-negative integers
$m\geq n$, there is an isomorphism of
$\bu$-modules
$$
\mathcal F_{m+n}\oplus \mathcal
F_{m+n-2}\oplus \cdots \oplus \mathcal F_{m-n}\cong\mathcal F_m\otimes
\mathcal F_n.
$$
Moreover the isomorphism is defined on highest weight vectors by
$$
\varPhi(u^{(m+n-2p)})=\sum_{k=0}^p(-1)^k
\frac{[n-p+k]![m-k]!}{[n-p]![m]!} v^{(k-p)(m-p-k+1)}u_k^{(m)}\otimes
u_{p-k}^{(n)}.
$$
\end{lem}
\begin{lem}  The map $\varphi:\mathcal F_m^{\rho_1}\to \mathcal F_m$ given
by $\varphi(u^{(m)}_k)=(-v)^{-k}u^{(m)}_{m-k}$ is an isomorphism.
\end{lem}

\begin{cor}\label{CG2} Let
$m\geq n$ be two non-negative integers. Then there is an isomorphism of
$\bu$-modules
$$
\mathcal F_{m+n}\oplus \mathcal
F_{m+n-2}\oplus \cdots \oplus \mathcal F_{m-n}\cong\mathcal F_m\otimes
\mathcal F_n^{\rho_1}.
$$
Moreover the isomorphism is defined on highest weight vectors by
\begin{align*}
\Phi(u^{(m+n-2p)})&=\sum_{k=0}^p(-1)^{n-p}\frac{[n-p+k]![m-k]!}{[n-p]![m]!}
v^{k\,\left( 2 + m -k\right)  + n - 2\,p - m\,p + p^2}u_k^{(m)}\otimes
u_{n-p+k}^{(n)} \\
\color{red}
&=(-1)^{n-p}v^{ n +p(p- 2 - m)}\sum_{k=0}^p 
  v^{k\,\left( 2 + m -k\right)} \qbinom{n-p+k}{ k}
   \qbinom{m}{ k}^{-1} u_k^{(m)}\otimes u_{n-p+k}^{(n)}.
\end{align*}
\color{black}

(the action on the second factor $u^{(n)}_l$ is twisted by the
automorphism $\rho_1$).
\end{cor}

\begin{lem}\label{CGspecialcase}\cite{MR2586983}. For $0\leq k\leq \min\{n-p,m+n-2p\}$,
\begin{align*}
&\left[\begin{matrix} m & n & m+n-2p
 \\ 0 & p+k & k \end{matrix}\right]= v^{-p(m-p+1)}{\qbinom {p+k}
{p}}\\ 
 \notag 
\end{align*}
and for $\max\{0,m-p\}\leq k\leq m+n-2p$
\begin{align*}
&\left[\begin{matrix} m & n & m+n-2p
  \\ m & p+k-m & k \end{matrix}\right]  \\
&\quad \quad =v^{p(p-1)-m(m +n  -p  - k)}
 \sum_{l=0}^{\min\{p,m\}}
(-1)^lv^{l\,\left( 1 + m + n - 2\,p - k \right)} 
{\qbinom {n-p+l} l} {\qbinom {p+k-m} {p-l}} \notag \\
&\quad \quad=(-1)^pv^{(p-m)(m +n) +m k}
 {\qbinom {m+n-p-k} {p}},\quad \text{if }\quad n\leq m.
\notag
\end{align*}
\end{lem}

Consider now the $\rho$-invariant forms \eqnref{normalizedforms} on
$\mathcal F_m$ and
$\mathcal F_n$, both denoted by $(,)$, normalized so that their
highest weight vectors have norm $1$.  Define the symmetric invariant
bilinear form on $\mathcal F_m\otimes \mathcal
F_n$ given by the tensor product of the two forms (the resulting
pairing is $\rho$-invariant).  In this case

\begin{align*}
(u^{(m+n-2p)},u^{(m+n-2p)})&=
\sum_{k=0}^p\left(\frac{[n-p+k]![m-k]!}{[n-p]![m]!}
v^{(k-p)(m-p-k+1)}\right)^2(u_k^{(m)},u_k^{(m)})(u_{p-k}^{(n)},u_{p-k}^{(n)})
\\ 
&= \sum_{k=0}^p\left(\frac{[n-p+k]![m-k]!}{[n-p]![m]!}\right)^2
v^{2(k-p)(m-p-k+1)+k^2-m k+(p-k)^2-n(p- k)}
{\qbinom m k} {\qbinom n {p-k}} \\ 
&= \sum_{k=0}^p\frac{[n]![n-p+k]![m-k]!}{[n-p]!^2[p-k]![k]![m]!}
v^{k\,\left( 2 + m + n - 2\,p \right)  + 
  p\,\left(3\,p -2 - 2\,m - n  \right)}  \\
&=
v^{p\,\left(3\,p -2 - 2\,m - n \right)}\frac{[n]!}{[n-p]!^2[m]!}
\sum_{k=0}^p\frac{[n-p+k]![m-p+(p-k)]!}{[k]![p-k]!}
v^{k\,\left( 2 + m + n - 2\,p \right)} \\
&=
v^{p\,\left(2\,p - 2\,m
-1\right)}\frac{[n]![m+n-p+1]![m-p]!}{[m]![p]![m+n-2p+1]![n-p]!}.
\end{align*}

where we have used formula \eqnref{Ma2}.

The same proof that gave us \eqnref{normalizedforms} now implies
\begin{equation*}
(u^{(m+n-2p)}_k,u^{(m+n-2p)}_k)=
v^{p\,\left(2\,p - 2\,m -1\right)-(m+n-2p-k) k}
\left[\begin{matrix} n \\  p
\end{matrix}\right]\left[\begin{matrix} m+n-p+1
 \\  p \end{matrix}\right]{\qbinom
{m+n-2p} k}\left[\begin{matrix} m \\  p
\end{matrix}\right]^{-1}.
\end{equation*}
\begin{prop}[\cite{MR2586983}]
\begin{enumerate}[(i).]
\item The basis $\{u_k^{(m+n-2p)}\}$
of \hbox{$\mathcal F_m\otimes{\mathcal F_n}$} is orthogonal.
\item For $0\leq i\leq m$, and $0\leq j\leq n$, 
\begin{align*}
u_i^{(m)}\otimes u_j^{(n)}
  &=v^{in +jm -2ij} 
 \sum_p v^{(1 + m - n - p) p}
  \frac{
     \left[\begin{matrix} m \\ i 
     \end{matrix}\right]\left[\begin{matrix} n \\
     j\end{matrix}\right] 
  \left[\begin{matrix} m
  \\  p \end{matrix}\right]\left[\begin{matrix} m & n & m+n-2p
  \\ i & j & {i+j-p} \end{matrix}\right]}
  {\left[\begin{matrix} n
  \\  p \end{matrix}\right]
  \left[\begin{matrix} m+n-p+1
  \\  p \end{matrix}\right]\left[\begin{matrix} m+n-2p
  \\  {i+j-p} \end{matrix}\right]}u_{i+j-p}^{(m+n-2p)}.
\end{align*}
\end{enumerate}
\end{prop}

In \hbox{$\mathcal F_m\otimes{\mathcal F_n}^{\rho_1}$} 
(recall $\phi(u^{(n)}_j)=(-v)^{-j} u^{(n)}_{n-j}$ )
\begin{align}
u^{(m)}_i\otimes u_j^{(n)}
&=(-1)^{j}
v^{i( 2j - n )     + mn-j(1 +m) } \\
 &\quad \times \sum_p v^{(1+m-n-p)p}
  \frac{\left[\begin{matrix} m \\ i 
     \end{matrix}\right]\left[\begin{matrix} n \\
     j\end{matrix}\right] 
     \left[\begin{matrix} m \\ p\end{matrix}\right]
   \left[\begin{matrix} m & n & m+n-2p
  \\ i & n-j & i+n-j-p \end{matrix}\right]}{
  \left[\begin{matrix} m+n-p+1 \\ p \end{matrix}\right]
  \left[\begin{matrix} n \\  p \end{matrix}\right]
  \left[\begin{matrix} m+n-2p
  \\  i+n-j-p \end{matrix}\right]}u_{i+n-j-p}^{(m+n-2p)}.\notag
\end{align}

In particular 
\begin{align}\label{CGdecomp}
u^{(m)}\otimes u_j^{(n)}&=(-1)^{j}
v^{ m n-j(1 +m)}  \sum_{p=0}^{\min\{n-j,m+j\}} v^{-np}
  \frac{\left[\begin{matrix} n \\  j \end{matrix}\right]
  \left[\begin{matrix} m \\ p\end{matrix}\right]
   \left[\begin{matrix}{n-j} \\ {p}\end{matrix}\right]}{
  \left[\begin{matrix} m+n-p+1 \\ p \end{matrix}\right]
  \left[\begin{matrix} n \\  p \end{matrix}\right]
  \left[\begin{matrix} m+n-2p
  \\ n-j-p \end{matrix}\right]}u_{n-j-p}^{(m+n-2p)} .
\end{align}
as 
\begin{equation*}
\left[\begin{matrix} m & n & m+n-2p
\\ 0 & p+r & r \end{matrix}\right]
=v^{(-p)(m-p+1)}{\qbinom {p+r} {p}}.
\end{equation*}
for $0\leq n-j-p\leq m+n-2p$, i.e. for $p\leq \{n-j,m+j\}$
\begin{align}
u^{(m)}_m\otimes u_j^{(n)}
&=(-1)^{j}
v^{ m(m+j - n) -j} \\
 &\quad \times \sum_p v^{(1+m-n-p)p}
  \frac{\left[\begin{matrix} n \\
     j\end{matrix}\right] 
     \left[\begin{matrix} m \\ p\end{matrix}\right]
   \left[\begin{matrix} m & n & m+n-2p
  \\ m & n-j & m+n-j-p \end{matrix}\right]}{
  \left[\begin{matrix} m+n-p+1 \\ p \end{matrix}\right]
  \left[\begin{matrix} n \\  p \end{matrix}\right]
  \left[\begin{matrix} m+n-2p
  \\  m+n-j-p \end{matrix}\right]}u_{m+n-j-p}^{(m+n-2p)}\\
&=(-1)^{j}
v^{-j} \\
 &\quad \times \sum_p (-1)^pv^{(1+m-p)p}
  \frac{\left[\begin{matrix} n \\
     j\end{matrix}\right] 
     \left[\begin{matrix} m \\ p\end{matrix}\right]
   \left[\begin{matrix} j \\ p \end{matrix}\right]}{
  \left[\begin{matrix} m+n-p+1 \\ p \end{matrix}\right]
  \left[\begin{matrix} n \\  p \end{matrix}\right]
  \left[\begin{matrix} m+n-2p
  \\  m+n-j-p \end{matrix}\right]}u_{m+n-j-p}^{(m+n-2p)},\notag
\end{align}
if $m\geq n$, as
\begin{align*}
&\left[\begin{matrix} m & n & m+n-2p
  \\ m & p+k-m & k \end{matrix}\right]  =(-1)^pv^{(p-m)(m +n) +mk}
 {\qbinom {m+n-p-k} {p}}.
\notag
\end{align*}

\section{Basis and the Intertwining map $\mathcal L$}
\subsection{A Basis} For $s\geq 1$,
and any lowest weight vector $\eta$ of weight $Tv^{\lambda+\rho}$, set
\begin{equation}\label{defofFinv}
F^{(-k)}\eta:=T'_{-1}(F^{(k)})\eta
=v^{k(k-1)}F^{-k}K^{k}[K;-1]^{(k)}\eta
=v^{k(\lambda+k)}T^k[T;\lambda]^{(k)}F^{-k}\eta
\end{equation}
as
$$
T'_{-1}(E^{(k)})=(-1)^kv^{-k(k-1)}K^{-k}F^{(k)},\quad
\quad T'_{-1}(F^{(k)})=(-1)^kv^{k(k-1)}E^{(k)}K^{k}.
$$

\begin{lem}[\cite{MR2586983}] \label{binomiallemma} Suppose $r,s\in\mathbb Z$, $s>
r\geq 0$, $\zeta$ is a highest weight vector of weight
$Tv^{\lambda-\rho}$ and $\eta$ is a lowest weight vector of
weight $Tv^{\lambda+\rho}$.  Then
\begin{equation}\label{hw1}
E^{(r)}F^{(s)}\zeta
=\qbinom{T;\lambda-1+r-s}{ r}F^{(s-r)}\zeta,
\quad F^{(r)}F^{(s)}\zeta
=\qbinom{r+s}{ r}F^{(s-r)}\zeta, 
\end{equation}
and
\begin{align}
F^{(r)}F^{(-s)}\eta
&=v^{r(\lambda+2s-r)}T^r
 \qbinom{T;\lambda+s}{ r} F^{(r-s)}\eta, \\
E^{(r)}F^{(-s)}\eta
&=(-1)^rv^{-r(\lambda+r+2s)}T^{-r}
 \qbinom{r+s}{ r}
  F^{(-r-s)}\eta, 
\end{align}
\end{lem}
Define indexing sets $ I_\lambda$ and $ I_{-\lambda}$ by $ I_\lambda =
\lbrace n-2,n-4,...\rbrace$, $ I_{-\lambda} = \lbrace -n,-n-2,...\rbrace. $  
One should compare the previous result with
\begin{lem}\label{locstructure}\cite[2.2]{MR2586983} Now for
integers $j \in  I_\lambda$ (resp.
$ I_{-\lambda}$) set $ k_j = {\frac{n-2-j}{ 2}}$ and $l_j = {\frac{-n-j}{2}}
$ and define basis vectors for $M(m+\l)$ and $M(m-\l)$ by 
$\mathbf v_j= F ^{k_j} \otimes 1_{m+\lambda}$ and $\mathbf v_{j,s}= F ^{l_j}\otimes 1_{m-\lambda
-\rho}$. The action of $\ru$ is given by 

\begin{align}
K\mathbf v_j&= Tv^{\lambda-1-2k_j}\mathbf v_j, \quad F \mathbf w_{\lambda,j} = \mathbf w_{\lambda,j-2} \\
K\mathbf v_{j,s}&=Tv^{-\lambda-1-2l_j}\mathbf v_{j,s}\quad  F\mathbf w_{\lambda,j} =\mathbf w_{\lambda,j-2} \\
E\mathbf v_j&=  [k_j][T; -l_j] \mathbf w_{\lambda,j+2}\quad ,\quad    E\mathbf v_j = [l_j][T; -k_j]
\mathbf w_{-\lambda,j+2}. 
\end{align}
\end{lem}

\begin{cor}\label{Lonbasis}  For $k\geq 0$  and $0\leq j\leq n$ we
have
\begin{align*}
\mathcal L&(F^{(-k)}\eta \otimes u_j^{(n)}) \\
&= \sum_{q=j-n}^{j}   
  \sum_{p=0}^{n-j+k}\sum_{t=q}^k\,
  (-1)^{t-q} v^{-\frac{3p(p-1)}{2}+p( 2j+ 2p- n ) }\{p\} \\ 
&\quad \times 
   v^{t( -1 + 2j - 2k - n - \lambda)  + 
      q(1 + 4k - p - 2q + 2\lambda )}  \\   
&\quad
\times
   T^{2q-t}{\qbinom {p-t+j} j}\qbinom{T;\lambda+k}{ t}
 \qbinom{k-q}{ t-q}\qbinom{n+q-j }{ p+q-t}
  F^{(q-k)}
   \eta\otimes  u^{(n)}_{j-q}.
\end{align*}
\end{cor}

\subsection{}  The
articles \cite{MR1327136} and
\cite{MR2586983} study noncommutative localization of highest
weight modules. This article may be viewed as an extension of what
was begun there. For any $\bu$-module $A$ let $A_F$ denote the
localization of
$A$ with respect to the multiplicative set in $\bu$ generated by $F$. If
$F$ acts without torsion on $A$ (we shall assume this throughout) then $A$
injects into $A_F$ and we have the short exact sequence of $\bu$-modules:
$0\ra A\ra A_F\ra A_\pi \ra 0$. 

\section{Maps into the Harmonics.}  
\subsection{Harmonics}  We know from \cite{MR1198203} that
$F(\bu)\cong \mathcal H\otimes Z(\bu)$ where $\mathcal
H=\oplus_{n\in\mathbb N}\mathcal L_{2n}$ is the space of {\it
harmonics}, and
$\mathcal L_{2n}\cong
\mathcal F_{2n}$.  We would now like to give an explicit basis of
$\mathcal L_{2n}$.

\begin{lem} For $n\in\mathbb N$, we can take
\begin{equation}
\mathcal H_{2n}=\oplus_{p=0}^{n}k(v)\text{ad}\,F^{(p)}(E^{(n)}K^{-n}),
\end{equation}
where
\begin{align}	
\text{ad}\,F^{(p)}(E^{(n)}K^{-n})
=\sum_{i=0}^n
 F^{(p-i)}\sum_{m=0}^{p-i}
  (-1)^{m+i}v^{( 1-p+2n )(m+i)  }
 \qbinom{ p-i }{ m} 
\qbinom{K; i-m-n}{ i}E^{(n-i)}K^{p-n} 
\end{align}
\end{lem}

For $m$, $n$, $r$ integers with $m+n$ even, $0\leq 2r\leq m+n$, we let
$\beta^{m,n}_{2r}:\mathcal F_m\otimes\mathcal F_n^{\rho_1}\to
\mathcal H$, be the
$\bu$-module homomorphism determined by 
\begin{equation}
\beta^{m,n}_{2r}(u^{(m+n-2q)})=\delta_{2r,m+n-2q}E^{(r)}
K^{-r}.
\end{equation}
so that $\text{im}\enspace\beta^{m,n}_{2r}=\mathcal L_{2r}$.

\begin{prop}\label{betaonabasis}Suppose $\eta$ is a lowest weight
vector of weight
$Tv^{\lambda+\rho}$.    Then
\begin{align*}
\rho_1&\Big(\beta^{m,n}_{2r}(u^{(m)}_i\otimes u^{(n)}_{j})
 \Big) F^{(-c)}\eta\\
&=(-1)^jv^{i(2j - n)+ mn-j(1 +m) +(1+ \frac{m-3n}{2}+r)
(\frac{m+n}{2}-r)} 
\\   \\
&\quad \times 
     \frac{\left[\begin{matrix} m \\  i \end{matrix}\right]
     \left[\begin{matrix} m \\  \frac{m+n}{2}-r \end{matrix}\right]
     \left[\begin{matrix} n \\  j \end{matrix}\right]
  \left[\begin{matrix} m & n & 2r
  \\ i & {n-j} & {i-j+\frac{n-m}{2}+r} \end{matrix}\right]}
  {\left[\begin{matrix} n
  \\  \frac{m+n}{2}-r \end{matrix}\right]
  \left[\begin{matrix} \frac{m+n}{2}+r+1
  \\  \frac{m+n}{2}-r \end{matrix}\right]\left[\begin{matrix} 2r
  \\  {i-j+\frac{n-m}{2}+r} \end{matrix}\right]} \\   \\
&\quad \times 
(v^{2\lambda+2+i-j+\frac{n-m}{2}+4c}T^2)^{j-i+\frac{m-n}{2}}\\   \\
&\quad \times 
 \sum_{l=0}^{i-j+\frac{n-m}{2}+r} (-1)^{r-l}
  v^{l\left(r+j-i+\frac{m-n}{2}+1 \right) }
 \qbinom{i-j+\frac{n-m}{2}+c}{ i-j+\frac{n-m}{2}+r-l}
 \qbinom{T;\lambda+l+c}{ r}
 \qbinom{l+c}{ l}
  F^{(j-i+\frac{m-n}{2}-c)}\eta
\end{align*}
and\begin{align*}
\rho_1&\Big(T_{-1}'\beta^{m,n}_{2r}(u^{(m)}_i\otimes u^{(n)}_{j})
 \Big) F^{(-c)}\eta\\
&=(-1)^jv^{i(2j - n)+ mn-j(1 +m) +(1+ \frac{m-3n}{2}+r)
(\frac{m+n}{2}-r)} \\   \\
&\quad \times 
     \frac{\left[\begin{matrix} m \\  i \end{matrix}\right]
     \left[\begin{matrix} m \\  \frac{m+n}{2}-r \end{matrix}\right]
     \left[\begin{matrix} n \\  j \end{matrix}\right]
  \left[\begin{matrix} m & n & 2r
  \\ i & {n-j} & {i-j+\frac{n-m}{2}+r} \end{matrix}\right]}
  {\left[\begin{matrix} n
  \\  \frac{m+n}{2}-r \end{matrix}\right]
  \left[\begin{matrix} \frac{m+n}{2}+r+1
  \\  \frac{m+n}{2}-r \end{matrix}\right]\left[\begin{matrix} 2r
  \\  {i-j+\frac{n-m}{2}+r} \end{matrix}\right]} (v^{ 2\,c
+\lambda}T)^{i-j+\frac{n-m}{2}}\\   \\
&\quad\times  \sum_{l=0}^{i-j+\frac{n-m}{2}+r}(-1)^{r-l}
  v^{  l (r +j-i+\frac{m-n}{2}+1) }
    \left[\begin{matrix}T;\lambda+c \\ l \end{matrix}\right]
    \left[\begin{matrix}T;\lambda+r+c-l \\ i-j+\frac{n-m}{2}+r-l \end{matrix}\right]
    \left[\begin{matrix}r+c-l \\ r \end{matrix}\right]
  F^{(i-j+\frac{n-m}{2}-c)}\eta  
\end{align*}
\end{prop}

\section{Symmetry Properties of Induced Forms}

\subsection{Twisted action of $R$.}  We shall twist by an automorphism
of 
$\ru$ in the setting of $\ru$-modules. Let $\Theta$ be an automorphism of
$\ru$. Then for any $\ru$-module
${_R\mathcal E}$ define a new $\ru $-module ${_R\mathcal E}$ with set equal to
that of ${_R\mathcal E}$ and action given by: for $e\in {_R\mathcal E}$ and $x\in\ru $
the action of $x$ on $e$ equals $\Theta(x)e$. For any
$\ru$-module $A$, let $A^{s\Theta}$ denote the module with action $\delta_i $ on
$A^{s\Theta}$ defined as follows: For
$a\in A$ and $x\in\ru$, 
$$ x\delta_i a= s\Theta(x)\ a.
$$
\smallskip

\begin{lem}
Suppose 
$\phi$ is a
$\brho$-invariant $R$-valued pairing of $\ru $-modules
$A$ and $B$.  Then $s\circ
\phi$ is a 
$\brho$-invariant pairing of
$A^s$ and
$B^s$ as well as $A^{s\tiepp}$ and $B^{sT_{-e}''}$. Also
$\phi$ itself is $\brho$-invariant pairing of these two pairs taking values
in the $R$-module $R^s$.
\end{lem}

%
%

{\begin{lem}
Let $m$ be a lowest weight vector in $M_\pi$ of weight
$Tv$ and $\Psi(m)$ a highest weight vector in $M$ of weight
$Tv^{-1}$.  The map
$\Psi:(M_\pi)^{s\Theta}\to M$ given by
$$
\Psi(F^{-k}m)=\frac{(-1)^kv^{-k^2}T^k}{[k]![T;-1]_{(k)}}F^{k}\Psi(m)
$$
for $k\geq 0$ is an isomorphism.
\end{lem}}
The above can be rewritten as 
$$
\Psi(F^{(-k)}m)=F^{(k)}\Psi(m).
$$

Set

\begin{equation}
\bar \Psi:=\Psi \otimes s\Theta\circ L^{-1}:(M_\pi\otimes \mathcal E)^{s\Theta}\to
M\otimes \mathcal E.
\end{equation}
Define $t\in \End(\mathbb P(M\otimes \mathcal E, N\otimes
\mathcal F))$ by
\begin{equation}\label{twistedpairing}
t(\chi)(\bar\Psi(a),\bar\Psi(b)):=
s\circ\chi_\pi\circ L(a\otimes b)
\end{equation}
for $a\in (M\otimes \mathcal E)_\pi$, $b\in (N\otimes \mathcal F)_\pi$
and $\chi\in\mathbb P(M\otimes \mathcal E,N\otimes \mathcal F)$. 
On the left hand side one is to consider
$a\in(M\otimes \mathcal E)_\pi^{s\theta}$, and
$b\in (N\otimes\mathcal  F)^{s \theta}_\pi$ and on the
right hand side $a\in (M\otimes \mathcal E)_\pi$, $b\in (N\otimes
\mathcal F)_\pi^{\rho_1}$.  We can view $L:(M\otimes \mathcal
E)_\pi^{s\theta}\otimes (N\otimes \mathcal
F)_\pi^{ s\theta\rho_1}\to  ((M\otimes \mathcal E)_\pi\otimes
(N\otimes \mathcal F)^{\rho_1}_\pi)^{s\theta}$ as a module isomorphism (see
\lemref{Tandrho} and 
\corref{Lisomorphism}) Then
$s\circ \chi_\pi\circ L:(M\otimes \mathcal
E)_\pi^{s\theta}\otimes (N\otimes \mathcal
F)_\pi^{s\theta\rho_1}\to 
R$ is a module homorphism.  More explicitly  we can show that $t(\chi)\in
\mathbb P(M\otimes \mathcal E, N\otimes
\mathcal F)$ by the following calculation for $x\in \bu$:

\begin{align*}
\sum\chi^\#(S(x_{(2)})\bar\Psi(a),\brho(x_{(1)})\bar\Psi(b))&=
\sum\chi^\#(\bar\Psi(\theta S(x_{(2)})a),\bar\Psi(
\theta \brho(x_{(1)})b)) \\ 
  &=\sum s\circ\chi_\pi\circ L(\theta S(x_{(2)})a\otimes 
\theta \brho(x_{(1)})b) \\ 
  &= s\circ\chi_\pi(\theta S(x)L(a\otimes b)) \\ 
  &= \mathbf e(x)s\circ\chi_\pi\circ L(a\otimes b)  \\
  &=\mathbf e(x)\chi^\#(\bar\Psi(a),\bar\Psi(b))\\
\end{align*}
where the third equality is from \corref{Lisomorphism}, and the fourth
equality is due to the fact that $\chi_\pi$ is $\rho$-invariant Note
that as a linear map
$L\in\text{End}\,((M\otimes
\mathcal E)_\pi\otimes(N\otimes
\mathcal F)_\pi)$ is well defined as $F$ acts locally nilpotently on $
(M\otimes \mathcal E)_\pi$.

\subsection{} For any homomorphism
$\beta$ of
$_R\mathcal E^{s\Theta}\otimes _R(\mathcal F^{s\Theta})^{\brho_1}$ into
$F(\bu)$, define another such
$t(\beta)$ by the formula
\begin{equation}
t(\beta) = s \Theta^{-1}\circ \beta \circ L\circ (s\Theta_{_R\mathcal
E}\otimes s\Theta_{_R\mathcal F}) 
\end{equation}
where the $\Theta^{-1}$ on the left is assumed to be the module
homomorphism defined on $F(\bu)$.

\begin{thm}
Let $\lambda=0$.  Suppose
$\chi$ is the induced pairing
$\chi_{\beta,\phi}$ as defined in \eqnref{inducedpairing}, $\mathcal F_m$
and
$\mathcal F_n$ are the $X$-graded finite dimensional $\bu$-modules
given in \secref{CG} and
$\phi$ is a $\bu$-invariant pairing satisfying $s\circ
\phi_{\pi}\circ L=\phi\circ (\Psi\otimes
\Psi)$. Then 
\begin{equation}\label{hardprop}
t(\chi _{\beta,\phi}) = \chi_{\beta ,\phi}\ .
\end{equation} 
\end{thm}

\subsection{}  Fix a finite
dimensional $\bu$-module $\mathcal F$  with highest weight $v^n$ and let $M$ be
the Verma module of highest weight $Tv^{-1}$.  Recall from \cite[\S 2]{MR1327136},  the
modules $P(m+\lambda):=P_{m+\lambda}$. Then we have the decomposition 
$M\otimes\mathcal F =\sum_iP(m+i)$ where the sum is over the nonnegative
weights of $\mathcal F$ and  by convention we set $P(m)=M(m)$. Set
$P_i=P(m+i)$ and following the notation of [E,3.6] let $\mathbb Z_i$ equal the
set of integers with the opposite parity to $i$. For $j\in \mathbb Z_\iota$ , set
$z^i_j = \mathbf v_j + \mathbf v_{j,s}$. Then for $i\in \mathbb N^*$, the set $\{ [T;0] \mathbf v_j :j\in
\mathbb Z_\iota\} \cup \{ z^i_j : j\in \mathbb Z_\iota\}$ is an $R$ basis for the 
localization $P_{i,F}$. Also the action of ${_R}\mathbf U_i$ is given by the formulas
in \lemref{locstructure} as well  as the formulas : for all indices
$j\in
\mathbb Z_\lambda $,
\begin{align}
 K_\mu \mathbf z_j&= T\, v^{\angles {\mu}{s\lambda-\rho-l_ji'}}\mathbf z_j ,\quad   F
\mathbf z_j = \mathbf z_{j-2}\label{eqn1} \\
E \mathbf z_j &= [l_j][T;-k_j] \mathbf z_{j+2} +[k_j-l_j][T;0]\mathbf  w_{\lambda ,j+2}.\label{eqn2}
\end{align}

The corresponding picture for $P_n$ of weight vectors for $P_n$ is given by 

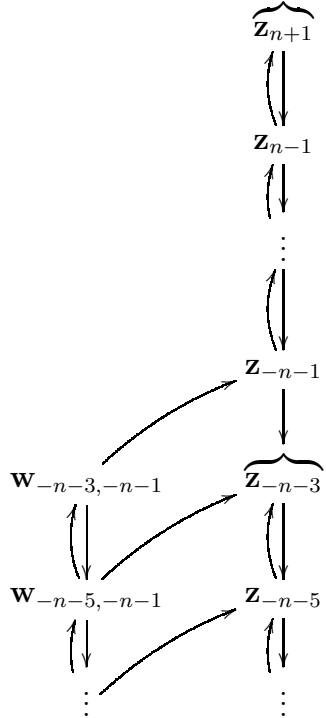
\[
\xymatrix{
& {\overbrace{\mathbf z_{n+1}}} \ar[d]\\
& \mathbf z_{n-1} \ar@/^/[u] \ar[d]\\
& {\vdots} \ar@/^/[u] \ar[d]\\
& \mathbf z_{-n-1}\ar@/^/[u] \ar[d]\\
\mathbf w_{-n-3,-n-1}\ar@/^/[ur] \ar[d] & {\overbrace{\mathbf z_{-n-3}}} \ar[d]\\
\mathbf w_{-n-5,-n-1} \ar@/^/[ur] \ar@/^/[u] \ar[d] & \mathbf z_{-n-5} \ar@/^/[u] \ar[d]\\
{\vdots} \ar@/^/[ur] \ar@/^/[u] & {\vdots} \ar@/^/[u] 
}
\]

Fix a positive weight $v^r$ of $\mathcal F$ and let $P=P_r$. Set 
$\germ L$ equal to
the
$m-r$th weight space of $P$. Then $\germ L$ is a free rank two $R$-module with
basis $\{\mathbf z^r_{-r-1},[T;0] \mathbf w_{r,-r-1}\}$. Define an $s$-linear map
$\G$ on $\mathfrak L$ and constants $a_{\pm r}$ by the formula: 
\begin{equation}
\G ([T;0] \mathbf w_{\epsilon r,-r-1}) = \Psi_F([T;0] \mathbf w_{\epsilon r,r+1}) = a_{\epsilon r}\ [T;0]
\mathbf w_{-\epsilon r,-r-1}.
\end{equation}

\def\ov{\overline}

This s-linear map $\G$ is the mechanism by which we analyze the symmetries
which arise through the exchange of
$\mathfrak L\cap ([T;0]\cdot M(m+r))$ and $\mathfrak L\cap ([T;0]\cdot M(m-r))$. The following is a
fundamental calculation for all which follows. Set $\overline{\G}=[r]!\
\G$.

\begin{lem} Let $\epsilon=\pm 1$.  For $a,b\in \mathfrak L\cap M(m+\epsilon r)$, we have: 
$$
\chi^\sharp(\G a,\G b)=\frac{1}{ T^r [r]![T^{-\epsilon};r]_{(r)}}\ s\chi(a,b)\quad
and \quad
\chi^\sharp(\ov\G a,\ov\G b)=u_\epsilon\ s\chi(a,b),
$$  where $u_\epsilon$ is a unit and $u_\epsilon\equiv 1 \mod (T-1)$. 

\end{lem}

\subsection{} Now we recall to the delicate calculation of the constants
$a_{\pm r}$.

\begin{lem}[\cite{MR2586983}] We may choose a basis for $P_r$ satisfying the relations
\eqnref{eqn1} and \eqnref{eqn2}, dependent only on the cycle $\Psi$, and such that the
constants $a_{\pm r}$ are uniquely determined by the three relations:
\begin{align}
 a_{-r}=s\ a_r, \quad a_r^2=\frac{1}{
[r]![T^{-1};r]_{(r)}}
\quad \text{and}
\quad a_r\equiv \frac{-1}{ [r]!}
\mod T-1.
\end{align}

\end{lem}
\begin{cor} For $\epsilon=\pm$,
$$
\overline\G ([T;0] \mathbf w_{\epsilon r,-r-1}] = w_{r,\epsilon} [T;0] \mathbf w_{-\epsilon r,-r-1}\ ,
$$  where
$w_{r,\epsilon}$ is the unit determined by conditions: 
\begin{equation}
w_{r,\epsilon}^2=\frac{[r]!}{ [T^{-\epsilon};r]_{(r)}}\quad \text{and}
\quad w_{r,\epsilon}\equiv -1-\epsilon\alpha (r)(T-1)\ \mod\ (T-1)^2\ .
\end{equation}
 Moreover $\ov\G$ induces a $\mathbb C$-linear map on $\mathfrak L/ \ (T-1)\cdot \mathfrak L$
given by the matrix
\begin{equation}
\begin{pmatrix}
 1&-\frac{\alpha (r)T(v-v^{-1})}{ (1+T)} \\  0&1
\end{pmatrix}
\end{equation}
where $\alpha(r)=-\sum_{s=1}^r
\frac{v^s+v^{-s}}{ v^s-v^{-s}}$.  Moreover, if
$x_\epsilon\in M(m+\epsilon r)$ and $[T;0]\cdot x_\epsilon$ is an $R$-basis vector for $\mathfrak L\cap
M(m+\epsilon r)$ then $\{[T;0]\cdot x_\epsilon,x_\epsilon +\overline \G x_ \epsilon\}$ is an $R$-basis
for $\mathfrak L$ and $x_\epsilon+\overline \G x_\epsilon$ generates the
$\ru $-submodule $P_r$.
\end{cor}

\subsection{} To verify the correct choice of sign for the third identity we
shall need some preliminary lemmas. Let $M^\prime$ denote the span of all
the weight subspaces of $M$ other than the highest weight space. Let $\delta$
denote the projection of $M\otimes_R \mathcal F$ onto
$w_{0,-1}\otimes \mathcal F$ with kernel $M'\otimes_R \mathcal F$. Define constants
$c_\pm$ by the relations:
 $\delta(w_{r,r-1})\equiv c_+ w_{0,-1}\otimes F^{(k)}f_n\mod \ M'\otimes
_R\mathcal F$  and
$\delta(w_{ -r,-r-1})\equiv c_- w_{0,-1}\otimes F^{(l)}f_n \mod 
\ M'\otimes_R\mathcal F$ where $n-2k-1=r-1$ and $n-2l-1=-r-1$. For any integer
$t$ set $z_t=w_{0,-1}\otimes F^{(t)} f_n$. In a similar fashion define the
projection
$\delta^\vee$ of $M_F\otimes _R\mathcal F$ onto $w_{0,1}\otimes _R\mathcal F$ with
kernel 
$M^\vee\otimes _R\mathcal F$ and $M^\vee$ equal to the span of all weight
subspaces in $M_F$ for weights other than $m+1$.

\begin{lem} [\cite{MR2586983}] Set 
$$
A_t=\frac{[n-k+t]_{(t)}T^tv^{-t^2}}{ [t]!\
[T^{-1};t]_{(t)}}
$$
then 
$$ w_{r,r-1}=c_+\sum_{0\le t\le k}A_t w_{0,-1-2t}\otimes F^{(k-t)}f_n.
$$
\end{lem}
\begin{proof} Note that $0=E\cdot w_{r,r-1}$ and solve the recursion
relations in $A_t$.
\end{proof}

\begin{lem} [\cite{MR2586983}] 
\begin{align}
\delta^\vee(w_{r,r+1})=c_+\ \frac{v^{(n-k)(k+1)}
[T^{-1};2k-n-1]_{(k)}}{ [T^{-1};k]_{(k)}} \ w_{0,1} \otimes F^{(k)}
f_n, \\
  -\frac{c_-}{ [l]!}\equiv \frac{c_+}{ [k]!}\mod T-1 \ \ \ 
\text{and}
\quad
\delta(\Psi w_{r,r+1})\equiv -c_-\ \frac{1}{ [r-1]!}\ z_l\mod\ \
(T-1)\cdot P_r. 
\end{align}
\end{lem}

We now return to the proof of the congruence. Since $\Psi(w_{r,r-1})=-a_r
w_{-r,-r-1}$  we can calculate the constant $a_r$ as the ratio of
$\delta(\Psi(w_{r,r-1}))$ and $\delta( w_{-r,-r-1})$.  From (4.4) we find the ratio
is congruent to $\frac{-1}{ [r-1]}\mod T-1$. This completes the proof
of Lemma 4.5.

\subsection{} Recall from \secref{automorphism} the category $_R\mathcal C_i$ and note
that any module $N$ in the category is the direct sum of generalized
eigenspaces for the Casimir element \cite{MR2586983} in the sense that 
$N=\sum N^{(\pm r)}$ where the sum is over $\mathbb N$ and $N^{(\pm r)}$
contains all highest  weight vectors in $N$ with weights $m+r-1$ and
$m-r-1$. Note that $N^{(\pm r)} $ need not be generated by its highest
weight vectors. The decomposition in (4.2), $M\otimes_R\mathcal F\equiv \sum
P_i$ where the sum is over the nonnegative weights of $\mathcal F_R$ is such a
decomposition. In this case $(M\otimes _R\mathcal F)^{(\pm i)}= P_i$.  The
Casimir element $\Omega_0$ of $_R\mathbf U$ by 
\begin{equation}\label{quantumCasimir}
\Omega_0 =F E + \frac{v\tilde K_i-2+v^{-1}\tilde K^{-1}
}{( v-v^{-1})^2}.
\end{equation}
   Let $N^{(r)}$ (resp. $N^{(-r)}$) denote the submodule of $N$ where the
Casimir element acts by the scalar 
\begin{align}
c(\lambda)&=\frac{v^{r-1}T-2+v^{-r+1}T^{-1}}{( v-v^{-1})^2
}=[\sqrt{T};(r-1)/2]^2\
\\
c(s\lambda)&= \frac{v^{-r+1}T-2+v^{r-1} T^{-1}}{ (
v-v^{-1})^2}=[\sqrt{T};(-r+1)/2]^2 . 
\end{align}

\subsection{} We now turn to the general case where $\mathcal F$ is a finite
dimensional ${_R}\mathbf U_i$-module but not   necessarily irreducible. We extend the
definition of the
$s$-linear maps $\G$ and $\overline \G$ defined in \eqnref{} as follows. Decompose
$M\otimes_R \mathcal F$  into generalized eigenspaces for the Casimir
$(M\otimes_R \mathcal F)^{(\pm r)}$ and let $\mathfrak L^r$ denote the $m-r-1$ weight
subspace of
$(M\otimes_R \mathcal F)^{(\pm r)}$. Then set $\mathbb L=\sum \mathfrak L^r$. Decompose
$\mathcal F=\sum \mathcal F_j$  into irreducible ${_R}\mathbf U_i$ modules.  Then $M\otimes_R
\mathcal F=\sum M\otimes_R \mathcal F_j$ and so we obtain $s$-linear  extensions
also denoted $\G$ and 
$\overline \G$ from $\mathbb L\cap (M\otimes_R \mathcal F_j)$ to all of $\mathfrak L$.

\begin{prop} [\cite{MR2586983}] Suppose $\phi$ is any invariant form on
$M\otimes_R \mathcal F$ with $\phi = \pm \phi^\sharp$. Let $\{w_j, j\in J\}$ be
an $R$-basis for the
 highest weight space of $(M\otimes_R \mathcal F)^{(0)}$ and
$\{u_i, i\in I\}$ a basis of weight vectors for the $E$-invariant weight
spaces of  weight $m+t$ for $t<-1$. Set
$M_j$ equal to the ${_R}\mathbf U_i$-module generated by $w_j$ and $Q_l$ the
${_R}\mathbf U_i$-module generated by $(T-1)^{-1}(u_l+
\G(u_l))$. Then $M\otimes_R \mathcal F=\sum_jM_j\oplus \sum_lQ_l$ where each
$M_j\cong M(m)$ and  if $U_l$ has weight $m-t$, then $Q_l\cong P(m+t)$.
Moreover, if the basis vectors
$w_j$ and $u_l$ are $\phi$-orthogonal then the sum is an orthogonal sum of
${_R}\mathbf U_i$-modules.

\end{prop}

\begin{prop}[\cite{MR2586983}]
Suppose $\phi$ is any invariant form on
$M\otimes_R \mathcal F$ with $\phi = \pm \phi^\sharp$. Then $M\otimes_R \mathcal F$
admits an orthogonal decomposition with each summand an indecomposible
$\mathfrak a$-module and isomorphic to $M$ or some $P(m+t)$ for $t\in\mathbb N^*$.
\end{prop}

\begin{proof}   Since $R$ is a discrete valuation ring we may choose
an orthogonal $R$-basis for the free $R$-module $\mathfrak L\cap (F_R\otimes
M)^{(+)}$.
\end{proof}

\section{Filtrations}

\subsection{} We continue with the notation of the previous section. So
$\phi$ is an invariant form on
$M\otimes_R \mathcal F=\sum_iP_i$. For any $R$-module $B$ set $\overline
B=B/(T-1)\cdot B$ and for any  filtration $B=B_0\supset B_1\supset
...\supset B_r$, let $\overline B=\overline B_0\supset \overline B_1\supset
...\supset 
\overline B_r$ be the induced filtration of $\overline B$, with $\overline
B_i=(\overline B_i+(T-1)\cdot B)/ (T-1)\cdot B$. 

Now $\phi$ induces a filtration on $F_R\otimes M$ by 
\begin{equation}
(M\otimes_R \mathcal F)^i=\{v\in  M\otimes_R \mathcal F| \phi(v,M\otimes_R \mathcal
F)\subset (T-1)^i\cdot R\}.
\end{equation}

\subsection{}
 Let $P=P_r$ and $\mathfrak L$ equal to the
$m-r$th weight subspace of $P$. Suppose $P=P_0\supset P_1\supset ...\supset
P_t=0$ is a filtration. Then  choose constants $a,b$ and $c$ so that $a$ is
maximal with $\overline P=\overline P_a$ , $b\ge a$ maximal with
$\overline P_{a+1}/\overline P_b$ finite dimensional if such exist and
otherwise set $b=a$and $c$ maximal with
$\overline P_c\neq 0$. We say that the filtration is of type $(a,b,c)$.

\begin{lem} Set $\phi(w_{r,-r-1},w_{r,-r-1})=p$ and
$\phi(w_{-r,-r-1},w_{-r,-r-1})=q$.  Then on $\mathfrak L$, $\phi$ is represented with
respect to the basis $\{w_{r,-r-1},(T-1)^{-1}(w_{r,-r-1}+w_{-r,-r-1})\}$ by
the matrix:

$$
\begin{pmatrix}
p&(T-1)^{-1}p\\ (T-1)^{-1}p& (T-1)^{-2}(p+q) 
\end{pmatrix}
$$  
with determinant $(T-1)^{-2}pq$. Suppose $p$ has  order $d$ (i.e.
$(T-1)^d$ divides $p$ and $(T-1)^{d+1}$ does not) and $q$ has order
$d^\prime$. Then if
$d\neq d^\prime$, the filtration is of type
$(min\{d,d^\prime\}-2,d-1,max\{d,d^\prime\})$ and  if $d=d^\prime$, the
filtration is either of type $(d-1,d-1,d-1)$ or type $(d-2,d-1,d)$
depending  as $p+q\equiv 0\ mod \ (T-1)^{d+1}$ or not.
\end{lem}

\subsection{}
Suppose $\phi^\sharp = \phi$ (resp. $-\phi$). In the first case
we say
$\phi$ is
$\mathbb Z_2$-invariant and in the second skew invariant. 

\begin{cor} Suppose $\phi^\sharp = \phi$ (resp. $-\phi$). and
other notation is as in (5.2). Then if
$d$ is even, $(P_r,\phi)$  has a filtration of type $(d-2,d-1,d)$ (resp.
$(d-1,d-1,d-1)$ ) and if $d$ is odd
$(P_r,\phi)$ has a filtration of type $(d-1,d-1,d-1)$ (resp.
$(d-2,d-1,d)$). 

\def\ov{\overline}
\end{cor}

\subsection{}
\begin{cor}
 Suppose $\mathcal F$ is a finite dimensional $\bui$-module
and $\phi$ is an invariant form on $M\otimes_R \mathcal F$.  Assume
$\phi^\sharp = \phi$ (resp.
$-\phi$) and let $M\otimes_R \mathcal F=B_0\supset B_1\supset ...\supset B_r$
be the filtration  (5.1.1) and
$\overline B_0\supset \overline B_1\supset ...\supset \overline B_r$ the
induced filtration on
$\overline{M\otimes_R \mathcal F}$.  Then
\begin{enumerate}
\item {(i)} The $\bui$-module $\overline B^i/\overline B^{i+1}$ is finite
dimensional for $i$ odd (resp. even). 
\item {(ii)} The $\bui$-module $\overline B^i/\overline B^{i+1}$ is both free and
cofree as a $ k(v)[F]$-module, for $i$ even (resp. odd).
\end{enumerate}
\end{cor}

\subsection{} The symmetry of $\mathbb Z_2$-invariant and skew forms can be
expressed in another form. Define the Jantzen sum of a filtration to be
$\sum_i character(\overline B_i)$. 

\begin{cor} Let notation and assumptions be as in (5.4). Then in
both cases the Jantzen sum of the filtration is invariant by the Weyl group
action on the characters which exchanges the characters of Verma modules
$M(t)$ and $M(-t)$ for all $t\in \mathbb Z$.

\end{cor}
\def\TT{(T-1)^{-1}\cdot}
\def\A{\mathbb A}
\def\B{\mathbb B}
\def\P{\mathbb P}
\def\L{\mathbb L}
\def\D{\mathbb D}
\def\oA{\overline{\mathbb A}}
\def\oB{\overline{\mathbb B}}
\def\oP{\overline{\mathbb P}}
\def\oL{\overline{\mathbb L}}
\def\oD{\overline{\mathbb D}}
\def\hra{\hookrightarrow} 
\section{ Filtrations and Wall-crossing } 

\subsection{}\label{P}
Here we begin the study of the relationships of wall-crossing
and the theory of induced and $\mathbb Z_2$-invariant forms. As in section
four set $M$ equal to the Verma module with highest weight
$Tv^{-1}$:i.e. $M=M(m)$. Then $M_\pi^{s\Theta}$ is isomorphic to $M$ itself and
so we may choose a cycle $\Psi : M_F\ra M$ which induces the isomorphism.
For $a\in M, x\in \ru, \ \Psi(sT_{i,-1}'(x)\cdot a)=x\cdot
\Psi(a).$ Recall that $\Psi$ is not $R$ linear due to the role of $s$ in
the previous formula. Let $\mathcal E$ denote the simple two dimensional
$\bui$-module and set $P=M\otimes_R\mathcal E$. Then $P$ is isomorphic to the 
basis module $P_1=P(m+1)$ as defined in \eqnref{eqn1} and \eqnref{eqn2}. Set $M_\pm=M(m\pm 1)$. Then
the construction of $P$ gives inclusions and a short exact sequence:
\begin{equation}
(T-1)\cdot M_+\oplus (T-1)\cdot M_- \subset P\subset M_+\oplus M_- 
,
\quad 0\ra (T-1)\cdot M_+\ra P\ra M_- \ra 0.
\end{equation}

Let $\mathcal F$ be a finite dimensional $\bui$-module and set $\A= (T-1)\cdot
M_+\otimes _R\mathcal F\ ,
\quad \B=(T-1)\cdot M_- \otimes _R\mathcal F
\ ,\quad \D=M_- \otimes _R\mathcal F\ ,\quad \P=P\otimes _R\mathcal F $. Then
(6.1.1) gives: 

\begin{equation}\label{ses3}
\A\oplus \B\subset \P \subset (T-1)^{-1}(\A\oplus \B),\quad 0\ra \A\ra
\P\ { \overset \Pi\to}\ \D\ra0 
\end{equation}

\subsection{} We now turn to the study of the $s$-linear map $\G$ defined in
(4.2.2) and derived from the cycle
$\Psi_\P=\Theta_{\mathcal F}\otimes \Theta_{\mathcal E}\otimes
\Psi$. Decompose $\P$ into generalized eigenspaces for the Casimir
$\P^{(\pm r)}$ and let $\mathfrak L^r$ denote the
$Tv^{-r-1}$ weight subspace of $\P^{(\pm r)}$. Recall $\L=\sum \mathfrak L^r$. From
(4.2.2) we obtain an $s$-linear map
$\G:\L\ra \L$.
\begin{lem}  We have the following
\begin{enumerate}[i).]
\item $\G$ restricts to an $s$-linear involutive isomorphism $\G\
:\L\cap\A\cong
\L\cap\B$. 
\item  $\G$ induces the identity map mod $T-1$; i.e. $\G(e)\equiv e$
mod $(T-1)\cdot \L$.
\item Suppose $\phi$ is any $\mathbb Z_2$-invariant (resp. skew
invariant) form on $\P$. Then for all $e\in
\L$, $\ \phi(\G e, \G e)=s\phi(e,e)$ (resp. $\ \phi(\G e, \G
e)=-s\phi(e,e)$). 
\end{enumerate}
\end{lem}
\subsection{} As before for any ${_R}\mathbf U_i$-module $N$ we let
$\overline N$ equal the quotient 
$N/T\cdot N$. From \eqnref{ses3} we obtain inclusions and a short exact sequence:

\begin{equation}
\oA \subset \oP,\quad\oB \subset
\oP,\quad 0\ra
\oA\to\oP\ \overset\Pi\to \oD\ra 0\ .
\end{equation}
Now fix an invariant form $\phi$ on $\P$ and let superscripts denote
the filtrations on $\P$,$\A ,\B$ and $\D$ induced by $\phi$ and the
restrictions of $\phi$ to $\A\times \A$ , $\B\times \B$ and $\D\times \D$
respectively. Let superscripts on $\oA,\oB$, $\oD$ and $\oP$ give the
filtrations obtained by projecting. So $\oA^i=\A^i/(T-1)\cdot
\A\cap
\A^i$,etc.

The ${_R}\mathbf U_i$-module $\overline P$ is of course also an $\bui$-module and it
has simple socle $M(-1)$. So we find that the inclusions $M_\pm\hra P$
induce an inclusion $\overline M_-\hra \overline M_+$. In turn we obtain
the inclusion $\oB\hra \oA$. 
\begin{prop} Let $\pi$ denote the natural map $\pi:\P\ra\oP$ and
suppose
$\phi$ is any $\mathbb Z_2$-invariant (resp. skew invariant) form on $\P$.
Then $\pi(\L\cap\A)=\pi(\L\cap\B)$ and for each
$i$, $\G$ induces an
$s$-isometry (resp. skew $s$-isometry)of $\L\cap \A^i$ onto $\L\cap \B^i$.

\end{prop}
\begin{cor} For all $i$, $\oA^i\cap \pi(\L) =\oB^i\cap \pi(\L)$.
\end{cor}
 For any form $\phi$ on $\P$ we say $\phi$ is {\it weakly
hereditary} whenever the identity of the corollary holds. We say $\phi$ is
{\it hereditary} if for all $i$, $\oA^i\cap \B=\oB^i$.

\subsection{}
The relationship between the filtrations of $\A$ and $\P$ is
more delicate than that between $\A$ and $\B$. 
\begin{lem}
 Assume that $\phi$ is induced from an invariant form on
$P$. Then $\A$ and $\B$ are orthogonal submodules,and $\A^i\subset
\P^{i-1}$ , $\B^i\subset \P^{i-1}$,  and $ \Pi(\P^i)\subset \D^{i-1}$. Also
$\D^i=(T-1)^{-1}\B^{i+2}$. 
\end{lem}

\begin{lem}
Assume that $\phi$ is induced from an invariant form
on $P$ and is $\mathbb Z_2$-invariant or skew invariant. Then 
$$
\L\cap\D^i\subset\Pi(\P^i)\quad \text{ and }\quad \oL\cap
\oP^i\cap\oA\subset \oA^i\ .
$$
\end{lem}

\section{Equivalence Classes of Forms}

\subsection{}
In this section we describe explicitly all the
equivalence classes of invariant forms on $B\otimes \mathcal E_R$ where 
$\mathcal E$ is the irreducible two dimensional ${_R}\mathbf U_i$-module and
$B$ is one of the indecomposible modules $M(\ru,m+b),\ b\in \mathbb Z$
or
$P(\ru,m+b),\ b\in \mathbb N^*$. Two invariant forms $\chi$ and
$\chi^\rho$ on a $\ru$-module $A$ are equivalent if there exists an
$\ru$-module automorphism $\kappa:A\ra A$ with $\chi(a,b)=\chi^\rho(\kappa
a,\kappa b)$. 

For $n\in \mathbb N$ let $\phi_{\pm n}$ denote the Shapovalov form on
$M({_R}\mathbf U_i,m\pm n)$ normalized by the identities:
\begin{equation}\label{normalization}
\phi_n(\mathbf v_{n,-n-1},\mathbf v_{n,-n-1})=1=\phi_{-n}(\mathbf v_{-n,-n-1},\mathbf v_{-n,-n-1})
\ .
\end{equation}

{\bf Caution}: This is not the obvious normalization. But we will
find it to be the most convenient.
\begin{lem}  The equivalence classes of invariant forms on
$M(\ru,m+n)$ (resp. $M(\ru, m-n)$ are represented by the forms
$\ \ T^r\cdot \phi_n,\ r\in \mathbb N^*$ (resp. 
$T^r\cdot \phi_{-n},\ r\in \mathbb N$).

\end{lem}

\subsection{}
The indecomposible module $P_n=P(\ru,m+n)$ was
constructed as a submodule $P_n\subset M_n\oplus M_{-n}$ where we set
$M_n=M(\ru,m+n)$ and $M_{-n}=M(\ru,m-n)$.  Therefore any
invariant form on $P$ is the restriction of the orthogonal sum of a
form on
$M_n$ and one on $M_{-n}$. For scalars $q$ and $r$ let
$\phi_{n,q,r}=q\phi_n\oplus r\phi_{-n}$ denote such an orthogonal sum
of forms.
\begin{lem} \label{lemma1.1}  The equivalence classes of invarant forms on
$P_n$ are represented by the degenerate forms:
$\phi_{_{n,T^b,0}}$ and $
\phi_{_{n,0,T^b}},\ b\in \mathbb N$ and the nondegenerate forms:
$\phi_{_{n,T^l,uT^k}}$ where $u$ is nonzero complex number and
$k,l\in
\mathbb N$.

\end{lem}
\begin{proof}  Recall the automorphism $\kappa$ of $P_n$ determined by
the units $u$ and $v$ with $u\equiv v$ mod $T$, as in the proof of
 \lemref{lemma1.1}. Then by
$\kappa$, we see that $\phi_{n,q,r}$ and $\phi_{n,u^2q,v^2r}$ are
equivalent. Choose integers $k$ and $l$ and units $u_0$ and $v_0$
with $q=u_0t^k$ and $r=v_0T^l$. Let $c$ be the complex number which
is the ratio of the constant term of $v_0$ by that of $u_0$. Then
$r=v_1cT^l$ with $v_1$ a unit and $U_0\equiv v_1$ mod $T$. Choose
square roots $u$ and $v$ with $u^2=u_0$,$v^2=v_1$ and $U\equiv v$
mod $T$. Then with $\kappa$ defined as above, we find that $
\phi_{n,q,r}$ is equivalent to $\phi_{n,T^k,cT^l}$. 
\end{proof}
\subsection{}  Suppose $\Psi$ is a cycle on $P_n$ and the basis is
chosen as in Lemma and Corollary 4.3. Set $v_\pm=v_{\pm n,-n-1}$.
Then $\ov\G v_\pm=v_\mp$.
\begin{lem}    Suppose $\chi$ is an invariant form on $P_n$
which is also $\mathbb Z_2$-invariant. Then, for some $q\in R$,
$\chi=\phi_{n,q,sq}$ and $\chi$ is equivalent to one of the $\mathbb Z_2$-invariant forms $\phi_{n,T^d,(-1)^dT^d}$, for $d\in \mathbb N$.
\end{lem}
\begin{proof}   Choose $q$ and $r$ with $\chi=\phi_{n,q,r}$. Then
$\mathbb Z_2$-invariance gives $r=sq$. Write $q=uT^d$ with $u$ a unit.
Then $r=(-1)^dsu\ T^d$. Let $\kappa$ be the automorphism of $P_n$ which
equals $u_1$ on $M(\ru,m+n)$ and $su_1$ on $M(\ru,m-n)$ where
$u_1^2=u$. Then via $\kappa$, $\phi_{n,q,r}$ is equivalent to
$\phi_{n,T^d,(-1)^dT^d}$.
\end{proof}

Let us consider the identity
 \begin{equation}
\Theta\circ \beta(a\otimes b)=\beta(L
(\Theta_{\mathcal E}a\otimes \Theta_{\mathcal E}b))\ .
\end{equation}
Here $\Theta$ on the left is assumed to be the module homomorphism.
The motivation for this comes from \cite[5.3.4]{MR1227098}.  Let $v\in (\mathcal
E\otimes \mathcal F)^m$ be a highest weight vector, then
$\beta(v)\in F(\bu)^m$ is a highest weight vector.  Then
\begin{align*}
\beta(L
\circ(\Theta_{\mathcal E}\otimes \Theta_{\mathcal F})(F^{(j)}v))
  & =\beta(\Theta_{\mathcal E\otimes \mathcal F}(F^{(j)}v)) =\beta((-1)^jv^{e(jh+j)}F^{(h)}v) \\
  &=(-1)^jv^{e(jh+j)}F^{(h)}\beta(v) =T_{-1}'\beta(v) 
\end{align*}
where $h+j=m$.

\begin{proof}  Fix $m\in \mathbb N^*$ and let $D$ denote the $T+b-2m$
weight space in $\mathcal E_R\otimes M(\ru,m+b)$. Set $\overline {\mathbf v}=\mathbf v_{b,b-1-2m}$.
Then $D$ has basis $\{d_+,d_-\}$ where $d_+=E\otimes F\overline{\mathbf v},b_-=FK\otimes \overline{\mathbf v}$ and the form $\delta\otimes \phi_b$ on $D$ is
given by:
\begin{equation}
\begin{pmatrix}
 -1&0\\0&-m(T+b-m)
\end{pmatrix}
\phi_b(\overline v,\overline v).
\end{equation}

Suppose $b\ne 0$. Then $\mathcal E_R\otimes M_b\equiv M_{b-1}\oplus M_{b+1}$
and if we look at the weight space $D$ for $m\ne b$ we conclude:
$\delta\otimes \phi_b=q\phi_{b-1}+r\phi_{b+1}$ where both $q$ and $r$
equal $\phi_b(\overline {\mathbf v},\overline{\mathbf v})$ times a unit. Inturn this gives the
formula for $b\ne 0$. For $b=0$ note that $\delta\otimes \phi_b$ is
$\mathbb Z^2$-invariant and since $\mathcal E_R\otimes M_0\equiv P(\ru,m+1)$,
$\delta\otimes \phi_0=\phi_{1,q,sq}$ for some $q\in R$. Write $q=uT^d$
for some unit $u$. To determine the integer $d$ we need only
compare the values of the forms on an $R$-basis vector for the
highest weight space. But $\delta\otimes \phi_0(e_+\otimes\mathbf  v_{0,-1},
e_+\otimes \mathbf v_{0,-1})=-\phi_0(\mathbf v_{0,-1},\mathbf v_{0,-1})=-1$.
\end{proof}

\section{Change of Coordinates for Induced Forms}

\subsection{} We continue with the notation from the earlier sections.. Set $M$ (resp. $M_\pm$) equal to
the $\ru$-Verma module with highest weight $T-1$ (resp. $T,T-2$).
Let $\phi$ and $\phi_\pm$ be the canonical invariant forms on these
Verma modules.
Let $\mathcal E $ (resp. $\mathcal F$) be an irreducible $\mathbf U_i$-module of dimension
$d+1$ (resp. $d+2$). Then as $\mathbf B$-modules we have the short exact
sequence:
\begin{equation}\label{nonsplit}
0\ra \mathcal E \otimes R_{T}\ra \mathcal F\otimes R_{T-1}\ra R_{T-d-1}\ra 0\ .
\end{equation}
Inducing up to $\ru$ we obtain:
\begin{equation}
0\ra \mathcal E_R \otimes M_+\ra \mathcal F_R\otimes M\ra M(\ru,T-d)\ra 0\ .
\end{equation}
From \cite{} we know that restriction from $\mathcal E_R\otimes M(\ru,T+b)$ to
$\mathcal E_R \otimes R_{T+b-1}$ gives an isomorphism of $\ru$-invariant forms
on $\mathcal E_R \otimes M(\ru,T+b)$ to $\mathbf U_0$-invariant forms on
$\mathcal E_R\otimes R_{T+b-1}$. Since \eqnref{nonsplit} splits as an $\mathbf U_0$-module, each
$\mathbf U_0$-invariant form on $\mathcal E_R \otimes R_{T}$ can be extended uniquely to
$\mathcal F_R\otimes R_{T-1}$ with it equaling zero on the weight space for
$T-d-1$. With this convention we obtain a map $\phi _{\mathcal E}$  from invariant
forms on $\mathcal E _R\otimes M_+$ to those on $ \mathcal F_R\otimes M$.

\subsection{}  Fix a nonnegative integer $d$ and suppose  $\beta$ is an
$\ru$-module homomorphism $\beta:\mathcal E _R^\Theta \otimes \mathcal E _R\to\ru$ with
image equal to the irreducible with highest weight vector $E^{(d)}$
and normalized so that $\beta(f_d \otimes f_{-d})=\text{ad}(F^{(d)})(E^{(d)})$.
Normalize the basis of $\mathcal E _R$ by setting $\overline f_{d-2j}=
\frac{1}{ j!}
f_{d-2j}$. Then $\Theta_{\mathcal E}(\overline f_{d-2j})=(-1)^d \overline f_{-d+2j}.$

\begin{lem}\label{lastlemma}
Let $\gamma$ be the invariant form on
$M_b=M(\ru,T+b+1)$ normalized by $\gamma(v_{b-1,b},v_{b-1,b})=1$ and 
$\chi_{\beta,\gamma}$ the induced form on
${\mathcal E}_R\otimes M_b$.
For
$0\le j\le d$, define complex constants by $\beta(\Theta_{\mathcal E}
f_{d-2j}\otimes f_{d-2j})=c_j\ \beta(f_d \otimes
f_{-d})$ and define $p_d\in \mathbf U_0$ by $p_d(T)=
\prod_{1\le t\le d} \qbinom{T;d-2t}{2t}$. Then with respect to the basis
$\overline f_{d,d-2j}\otimes 1$, the restriction of the induced form
$\chi_{\beta,\phi}$ to $\mathcal E_R\otimes R_{T+b}$ is given by the diagonal
matrix 
\begin{equation}\label{matrixexp}
\begin{pmatrix}
c_0& & \\ &\ddots & \\ & & c_d
\end{pmatrix} p_c(T+b).
\end{equation}

Moreover, the constants have the symmetry:
$c_{d-j}=(-1)^dc_j,\ 0\le j\le d$.

\end{lem}
\begin{proof}  Let $\pi$ be the
projection of ${_R}\mathbf U_i$ onto $\mathbf U_0$ which is zero on  
$F\cdot \mathbf U_i+\mathbf U_i\cdot E$. First we evaluate $\pi
\beta(f_d\otimes f_{-d})$. Since
\begin{equation}
\text{ad} F^{(p)}(E^{(d)}_j)=E^{(d)}_{j+p},\end{equation}
by \eqnref{irred},
we conclude using \eqnref{hw1}
\begin{equation}
\text{ad}(F^{(d-2j)})E^{(d-2j)}
=\qbinom{T;d-2j}{ d}E^{(d-2j)},
\end{equation}
that
\begin{align}
\beta\pi(\text{ad}(F^{(d-2j)})(E^{(d-2j)}))=
\qbinom{T;d-2j}{ d-2j}
\end{align}

Now to obtain the matrix expression in \eqnref{matrixexp} we observe
\begin{align*}
\chi_{\beta,\gamma}(\overline f_{d-2j}\otimes 1,\overline f_{d-2j}\otimes 1)= 
\beta(\Theta_{\mathcal E}\overline f_{d-2j}\otimes \overline f_{d-2j})\cdot 1_{T+b}
&=c_j
\prod_{1\le t\le d} \qbinom{T;d-2t}{2t}.
\end{align*}

Finally regarding the symmetry, $\beta\circ (\Theta_{\mathcal E}\otimes \Theta_{\mathcal E})=\Theta \circ \beta$. So by the basis given in \secref{basisandsymmetries}, gives us
\begin{align*}
(-1)^dc_j\beta(\overline f_d\otimes \overline f_{-d})&=
\Theta\beta(\Theta_{\mathcal E} \overline f_{d-2j}\otimes \overline f_{d-2j})= \beta(\Theta_{\mathcal E}^2\overline f_{d-2j}\otimes \Theta_{\mathcal E} \overline f_{d-2j}) \\
&=\beta(\Theta_{\mathcal E}\overline f_{-d+2j}\otimes  \overline f_{-d+2j})=c_{d-j} \beta(\overline f_d\otimes \overline f_{-d}).\notag
\end{align*}
This gives the symmetry of the constants and completes
the proof.
\end{proof}

Let $\mathbb D$ denote the diagonal matrix with constants $c_j$ on the
diagonal. Then \lemref{lastlemma} asserts that the induced form
$\chi_{\beta,\gamma}$ is determined by the matrix $\mathbb D\cdot p_d(T+b)$.
We shall call $\mathbb D$ the coefficient matrix associated to 
$\beta$ and  $\mathbb D\cdot p_d(T+b)$ the full matrix associated to the
form $\chi_{\beta,\gamma}$.

We plan on describing the coefficient matrix for particular cases in a future publication. 

\section{Aknowledgements} The first author has been supported in part by a Simons Collaboration Grant \#319261. The second author has been supported in part by NSF Grant
DMS92-06941.

\def\cprime{$'$}

%
%

\end{document}